\documentclass[a4paper]{amsart}
\usepackage{fullpage}
\usepackage{amsmath,amssymb,mathtools,booktabs,tikz}
\usetikzlibrary{arrows.meta,calc,decorations.markings,positioning}
\usepackage[T1]{fontenc}
\usepackage{lmodern}
\usepackage{microtype}
\usepackage{needspace}
\usepackage{tabularx,array}
\usepackage[hypertexnames=false]{hyperref}
\usepackage{doi}
\usepackage{float}
\hypersetup{colorlinks=true,linkcolor=blue,citecolor=blue,urlcolor=blue,
  pdftitle={On the embedding rigidity problem for uniformly locally finite coarse spaces},pdfauthor={Teng Zhang}}

\newtheorem{thm}{Theorem}[section]
\newtheorem{lem}[thm]{Lemma}
\newtheorem{prop}[thm]{Proposition}
\newtheorem{cor}[thm]{Corollary}
\newtheorem{prob}[thm]{Problem}
\theoremstyle{definition}
\newtheorem{defin}[thm]{Definition}

\numberwithin{equation}{section}

\newcommand{\Cu}{C_u^*}
\newcommand{\cK}{\mathcal K}

\DeclareMathOperator{\supp}{supp}

\DeclareMathOperator{\girth}{girth}
\DeclareMathOperator{\asdim}{asdim}

\title[Embedding rigidity for uniform Roe algebras]
{On the embedding rigidity problem for uniformly locally finite coarse spaces}
\author[T.~Zhang]{Teng Zhang}
\address{School of Mathematics and Statistics, Xi'an Jiaotong University,
Xi'an 710049, P. R. China}
\email{teng.zhang@stu.xjtu.edu.cn}
\date{October 5, 2026}

\begin{document}
\begin{abstract}
In this paper, we construct countable uniformly locally finite metric spaces
\(X\) and \(Y\) such that \(\Cu(X)\) is isomorphic to a hereditary
\(C^*\)-subalgebra of \(\Cu(Y)\), but neither space coarsely embeds
into the other.  This gives a negative answer to the embedding rigidity
problem for uniformly locally finite coarse spaces. On the positive side, we prove that a hereditary
embedding \(\Cu(X)\to\Cu(Y)\) implies an injective coarse embedding
\(X\to Y\) whenever \(\Cu(Y)\) contains no block-rank-one ghost projection.
In particular, this strengthens the coarse embedding theorem of Braga,
Farah, and Vignati~\cite{BFV20} by obtaining injectivity when every sparse
subspace of \(Y\) yields only compact ghost projections.
\end{abstract}
\maketitle

\section{Introduction}

Uniform Roe algebras are fundamental objects that encode the large-scale
geometry of discrete spaces within the framework of \(C^*\)-algebras.
Their origins can be traced back to John Roe's early work
on index theory for open manifolds~\cite{Roe88,Roe93}; for a modern treatment, see
Roe~\cite{Roe03}. Since then, they have been studied
extensively because of their deep connections with geometric group theory
and higher index theory, which have led to important applications in
manifold topology and geometry \cite{Roe96}. More recently, uniform Roe
algebras have also found applications in the study of topological phases
of matter \cite{Kub17} and, more broadly, in mathematical physics
\cite{BE26}.

The rigidity problem asks which large-scale properties of a coarse space \(X\)
are encoded by its uniform Roe algebra. Its study has been motivated in part
by the Baum--Connes and Novikov conjectures \cite{HR95,Roe96,BCL20}.
Recent preprints of OpenAI~\cite{OAI26a,OAI26b} report counterexamples to
rational injectivity and surjectivity, respectively, of the coefficient-free
reduced Baum--Connes assembly map. Bijective
coarse equivalences induce isomorphisms of uniform Roe algebras, while
injective coarse embeddings induce embeddings onto hereditary
\(C^*\)-subalgebras. More generally, coarse equivalences and coarse
embeddings induce the corresponding statements after stabilization.
The first fundamental question concerns isomorphism rigidity: if
\(C_u^*(X,\mathcal E)\) and \(C_u^*(Y,\mathcal F)\) are isomorphic, must
\((X,\mathcal E)\) and \((Y,\mathcal F)\) be bijectively coarsely equivalent?
For rigidity results on bounded-geometry metric spaces, see
\cite{SW13,BBFKVW22,Vig26,Kru26}; for arbitrary uniformly locally finite
coarse spaces, see \cite{BF21,BFV22}; for Cartan rigidity, see
\cite{WW20}; and for \(C^*\)-rigidity of Roe algebras, see
\cite{MV25}. Very recently, the author~\cite{Zha26} resolved this problem
for arbitrary uniformly locally finite coarse spaces by establishing
bijective coarse equivalence. Another important question is the following
embedding rigidity problem, posed by Braga, Farah, and
Vignati~\cite[Problem~1.2]{BFV22}.

\begin{prob}[Embedding rigidity problem]\label{prob:main}
Let \((X,\mathcal E)\) and \((Y,\mathcal F)\) be uniformly locally finite
coarse spaces. If \(C_u^*(X,\mathcal E)\) is isomorphic, as a
\(C^*\)-algebra, to a hereditary \(C^*\)-subalgebra of
\(C_u^*(Y,\mathcal F)\), must \((X,\mathcal E)\) admit a coarse
embedding into \((Y,\mathcal F)\)? Must there be an injective coarse
embedding?
\end{prob}

The first purpose of this paper is to show that the answer to
Problem~\ref{prob:main} is negative. More precisely,
we construct countable uniformly locally finite metric spaces \(X\) and \(Y\)
for which \(\Cu(X)\) is isomorphic to a hereditary \(C^*\)-subalgebra of
\(\Cu(Y)\), but \(X\) does not coarsely embed into \(Y\).
The construction also gives the following additional properties. We write
\(\asdim\) for asymptotic dimension.

\begin{thm}\label{thm:main}
There exist countable uniformly locally finite metric spaces $X$ and $Y$,
with $\asdim X=1$ and no coarse embedding in either direction, for which
$\Cu(X)$ is isomorphic to a corner $q\Cu(Y)q$ consisting entirely of
ghost operators, where $q$ is a projection in $\Cu(Y)$.
\end{thm}

Braga, Farah, and Vignati~\cite[Theorem~1.4(i)]{BFV20}
proved that if every sparse subspace of \(Y\) yields only compact
ghost projections, then an isomorphism of \(\Cu(X)\) onto a
hereditary \(C^*\)-subalgebra of \(\Cu(Y)\) implies that \(X\)
coarsely embeds into \(Y\). Under the stronger assumption that
\(Y\) has property~A, Braga, Farah, and
Vignati~\cite[Theorem~1.4(ii)]{BFV20} obtained an injective coarse
embedding. We obtain injectivity under the assumption that
\(\Cu(Y)\) contains no block-rank-one ghost projection, as defined
in Definition~\ref{def:block-rank-one}.

\begin{thm}\label{thm:compact-ghost-rigidity}
Let \(X\) and \(Y\) be uniformly locally finite metric spaces, and let
\(\Phi\colon\Cu(X)\longrightarrow\Cu(Y)\) be an injective
\(*\)-homomorphism with hereditary range. Suppose that \(\Cu(Y)\)
contains no block-rank-one ghost projection. Then there exists an injective
coarse embedding \(f\colon X\longrightarrow Y\).
\end{thm}

We will see that Lemma~\ref{lem:sparse-implies-block-rank-one} shows that the hypothesis
of Theorem~\ref{thm:compact-ghost-rigidity} holds whenever every
sparse subspace of \(Y\) yields only compact ghost projections.
Whether the converse implication holds is discussed by Braga,
Farah, and Vignati~\cite[Section~6]{BFV24}.
We therefore obtain the following strengthening of their
embedding theorem~\cite[Theorem~1.4(i)]{BFV20}.

\begin{cor}\label{cor:sparse-ghost-rigidity}
Let \(X\) and \(Y\) be uniformly locally finite metric spaces. Suppose that
every sparse subspace of \(Y\) yields only compact ghost projections. If
\(\Cu(X)\) is isomorphic to a hereditary \(C^*\)-subalgebra of \(\Cu(Y)\),
then \(X\) admits an injective coarse embedding into \(Y\).
\end{cor}

Since every ghost projection on a subspace of \(Y\) extends by zero to a
ghost projection on \(Y\), we obtain the following immediate consequence.

\begin{cor}\label{cor:global-ghost-rigidity}
Let \(X\) and \(Y\) be uniformly locally finite metric spaces. Suppose that
every ghost projection in \(\Cu(Y)\) is compact. If \(\Cu(X)\) is isomorphic
to a hereditary \(C^*\)-subalgebra of \(\Cu(Y)\), then \(X\) admits an
injective coarse embedding into \(Y\).
\end{cor}

We now outline the main ideas of the proofs.

\medskip
\noindent\textbf{Proof strategy and relation to previous work.}
For Theorem~\ref{thm:main}, we take \(X\) to be a coarse disjoint
union of cycles. Over each vertex of a cycle, we place a copy
of a finite connected regular graph, and we join the copies
over adjacent vertices by perfect matchings.
The results of Friedman~\cite{Fri08} and McKay, Wormald, and
Wysocka~\cite{MWW04} provide graphs with arbitrarily large girth
and a common positive lower bound on the nonzero Laplacian
eigenvalues. Lemma~\ref{lem:matching-completion} allows us to
add the matchings while preserving the required girth.

Sending each basis vector to the normalized constant vector
on the corresponding fiber defines an isometry
\(U\colon\ell^2(X)\to\ell^2(Y)\).
The spectral gap gives finite-propagation approximations to
\(q=UU^*\), following the construction of ghost projections
from expanders discussed by Roe and Willett~\cite[Section~3]{RW14}.
To obtain a hereditary embedding, we prove the identity
\[
q\Cu(Y)q=U\Cu(X)U^*.
\]
A related construction appears in Braga, Farah, and
Vignati~\cite[Proposition~5.5]{BFV24}, where block-rank-one
ghost projections give rise to hereditary embeddings for
coarse disjoint unions of singletons.
The cycles in our construction provide the obstruction to
coarse embeddability needed to answer
Problem~\ref{prob:main}.
The choice of girths and component sizes ensures that a coarse
embedding \(X\to Y\) would map a sufficiently long cycle into
a tree. Lemma~\ref{lem:tree-median} then gives a contradiction.
A coarse embedding in the reverse direction is impossible
because \(X\) embeds bi-Lipschitzly into Hilbert space, whereas
the spectral gaps of the fibers prevent \(Y\) from coarsely
embedding into Hilbert space.
Finally, every operator in the corner is a ghost, and \(q\)
is a block-rank-one projection.

For Theorem~\ref{thm:compact-ghost-rigidity}, we use the spatial
implementation and the estimates relating matrix coefficients
to coarse embeddings established by Braga, Farah, and
Vignati~\cite[Lemmas~5.3, 6.1, and~6.2]{BFV20}.
Their coefficient estimate~\cite[Lemma~3.5(2)]{BFV24}
allows us to choose a coarse embedding \(f\colon X\to Y\)
along uniformly large coefficients of the implementing isometry.
The uniform approximation result of Braga and
Farah~\cite[Lemma~4.9]{BF21}, in the form stated by Braga,
Farah, and Vignati~\cite[Lemma~2.3]{BFV24}, provides a common
propagation bound for approximations to all images of diagonal
projections. Averaging over finitely many choices of signs,
we prove that the balls \(B_Y(f(x),R)\), \(x\in X\), satisfy
Hall's condition for a radius \(R\) independent of the finite
subset under consideration. Hall's theorem then gives an
injective map at uniformly bounded distance from \(f\).

The use of Hall's theorem in this setting has precedents in
Braga, Farah, and Vignati~\cite[Lemmas~6.8 and~6.10]{BFV20},
White and Willett~\cite{WW20},
Vignati~\cite[Section~1.1]{Vig26}, and the author's
work~\cite[Section~7]{Zha26}.
Here the coefficient bound and uniform approximation of the
diagonal projection sums suffice to obtain the required
radius. The same argument applies to rank-preserving
embeddings without assuming that their ranges are hereditary;
see Corollary~\ref{cor:spatial-injective-coarse-map}.

\medskip
\noindent\textbf{Organization of this paper.}
Section~\ref{sec:prelim} introduces the notation and definitions.
Section~\ref{sec:bundles} constructs the finite graphs used in
the example, and Section~\ref{sec:obstruction} proves that
neither of the resulting spaces coarsely embeds into the other.
Section~\ref{sec:corner} identifies the hereditary corner and
completes the proof of Theorem~\ref{thm:main}.
Section~\ref{sec:positive} proves
Theorem~\ref{thm:compact-ghost-rigidity} and its corollaries.

\medskip
\noindent\textbf{Acknowledgements.}
The author is deeply grateful to his advisor, Professor Zongben Xu, for
his kind support and concern regarding both his academic work and personal
life. The author would also like to thank Professor Lajos Moln\'ar for his
generous assistance and hospitality during the author's visit to the
University of Szeged in 2026. This work was supported by the China
Scholarship Council, the Young Elite Scientists Sponsorship Program for
PhD Students (China Association for Science and Technology), and the
Fundamental Research Funds for the Central Universities at Xi'an Jiaotong
University (Grant No.~xzy022024045).

\medskip
\noindent\textbf{Disclosure of AI tools use.}
No AI tools were used in the original version of the manuscript. In the revised version, ChatGPT Astra 6.0 was used during the preparation of this paper to aid with proofreading, editing, creating
figures, and content critique. All mathematical ideas are due to the author.

\section{Preliminaries}\label{sec:prelim}

Except in Problem~\ref{prob:main}, the spaces underlying
the uniform Roe algebras considered here are discrete metric
spaces with finite-valued metrics. We write
$B_Z(z,r)=\{w\in Z:d_Z(z,w)\leq r\}$ and
$d_Z(A,B)=\inf\{d_Z(a,b):a\in A,b\in B\}$, with
$\inf\varnothing=+\infty$. A space $Z$ is \emph{uniformly locally finite},
or has \emph{bounded geometry}, if $\sup_{z\in Z}|B_Z(z,r)|<\infty$ for
every $r<\infty$. Its bounded coarse structure consists of the subsets
of $Z\times Z$ on which distance is bounded. Every nonempty uniformly
locally finite metric space is countable, since it is a union of the
finite balls of integer radius about any fixed point.

Inner products are linear in the first variable. The vectors
$(\delta_z)_{z\in Z}$ denote the standard orthonormal basis of
$\ell^2(Z)$. We write $B(\ell^2(Z))$ for the bounded operators and
$\cK(\ell^2(Z))$ for the compact operators, and $\sigma(T)$ for the
spectrum of a bounded operator $T$. For $A\subseteq Z$,
$\chi_A$ is the orthogonal projection onto $\ell^2(A)$; we identify
$\ell^\infty(Z)$ with the diagonal operators. The matrix unit $e_{zw}$
sends $\delta_w$ to $\delta_z$ and annihilates the other basis vectors.
All sums of mutually orthogonal projections are taken in the strong
operator topology.

For $a\in B(\ell^2(Z))$, define
\[
 \operatorname{prop}(a)
 =\sup\{d_Z(z,w):\langle a\delta_w,\delta_z\rangle\ne0\},
 \qquad
 \Cu(Z)=\overline{\{a\in B(\ell^2(Z)):
             \operatorname{prop}(a)<\infty\}}^{\|\cdot\|},
\]
where the propagation of zero is zero. Finite-support matrices are
norm dense in the compact operators, so
\begin{equation}\label{eq:compacts-in-roe}
 \cK(\ell^2(Z))\subseteq\Cu(Z).
\end{equation}

We use the following convention for coarse disjoint unions.

\begin{defin}\label{def:coarse-disjoint-union}
Let $(Z_n,d_n)$ be finite metric spaces, choose $o_n\in Z_n$, and let
$0=t_1<t_2<\cdots$ satisfy $t_{n+1}-t_n\to\infty$. The \emph{coarse
disjoint union} $Z=\bigsqcup_n Z_n$ retains $d_n$ on $Z_n$ and has
cross-component distances
\begin{equation}\label{eq:metric}
    d_Z(z,w)=d_n(z,o_n)+|t_n-t_m|+d_m(o_m,w)
 \quad(z\in Z_n,\ w\in Z_m,\ n\ne m).
\end{equation}
\end{defin}

It is easy to see that \eqref{eq:metric} defines a metric. 
For every \(r<\infty\), only finitely many pairs of distinct
components have distance at most \(r\), since
\(t_{n+1}-t_n\to\infty\). Their union is finite.
Consequently, a coarse disjoint union of finite graphs of
uniformly bounded degree has bounded geometry.
The coarse structure is independent of the basepoints and
the sequence \((t_n)\): any two such metrics agree within each
component, and, at each fixed scale, the pairs of points in
distinct components form a finite set.
The identity map is therefore coarse in both directions.

A map $f\colon X\to Y$ is \emph{coarse} if distances between images of
pairs at bounded distance are uniformly bounded. It is a \emph{coarse
embedding} if there are nondecreasing functions
$\rho_-,\rho_+\colon[0,\infty)\to[0,\infty)$, with
$\rho_-(t)\to\infty$, such that
\[
 \rho_-(d_X(x,z))\leq d_Y(f(x),f(z))\leq\rho_+(d_X(x,z))
 \qquad(x,z\in X).
\]
Equivalently, $f$ is coarse and, for each $r<\infty$, there is
$s<\infty$ such that
\begin{equation}\label{eq:effective-proper}
 d_Y(f(x),f(z))\leq r\quad\Longrightarrow\quad d_X(x,z)\leq s.
\end{equation}
Injectivity is not required. If $X$ has bounded geometry,
\eqref{eq:effective-proper} at $r=0$ gives a uniform finite bound on
the cardinalities of the fibers of $f$.

A family of subsets is \emph{$r$-disjoint} if distinct members are at
distance greater than $r$, and \emph{uniformly bounded} if their
diameters have a common finite bound. We write $\asdim Z\leq m$ if,
for every $r>0$, the space $Z$ has a uniformly bounded cover which is
the union of $m+1$ $r$-disjoint families. The asymptotic dimension is
the least such $m$, or infinity if none exists.

\begin{defin}\label{def:hereditary}
A $C^*$-subalgebra $B$ of a $C^*$-algebra $A$ is \emph{hereditary} if
$0\leq a\leq b$, $a\in A$, and $b\in B$ imply $a\in B$.
\end{defin}

A projection $q\in A$ determines the hereditary subalgebra $qAq$.
Conversely, a hereditary subalgebra with unit $q$ equals $qAq$:
for $a\in A_+$, the inequality $0\leq qaq\leq\|a\|q$ gives
$qaq\in B$, and taking linear spans proves $qAq\subseteq B$.
The reverse inclusion follows from $b=qbq$ for $b\in B$.

\begin{defin}\label{def:ghost}
An operator $T\in\Cu(Z)$ is a \emph{ghost} if, for every
$\varepsilon>0$, there is a finite $F\subseteq Z$ such that
$|\langle T\delta_w,\delta_z\rangle|<\varepsilon$ whenever
$(z,w)\notin F\times F$. A space \emph{yields only compact ghost
projections} if every ghost projection in its uniform Roe algebra is
compact.
\end{defin}

Equivalently, it suffices to require the coefficient estimate
when both \(z\) and \(w\) lie outside a finite set \(F\).
To see this, fix \(\varepsilon>0\) and choose such an \(F\).
Each row and column of a bounded operator on \(\ell^2(Z)\)
is square summable. Since \(F\) is finite, there is a finite
set \(E\supseteq F\) such that every coefficient with one
index in \(F\) and the other outside \(E\) has absolute value
less than \(\varepsilon\).
If \((z,w)\notin E\times E\), then either both indices lie
outside \(F\), or one lies in \(F\) and the other outside \(E\).
The required estimate holds in either case.

A subset $Z\subseteq Y$ is \emph{sparse} if it admits a partition
$Z=\bigsqcup_{n\geq1}Z_n$ into finite nonempty sets with
$d_Y(Z_n,Z_m)\to\infty$ as $n+m\to\infty$, $n\ne m$.
Compression to $\ell^2(Z)$ and extension by zero both preserve
finite propagation, norm approximation, and the ghost condition.

We use the terminology of Braga, Farah, and Vignati~\cite{BFV24}.

\begin{defin}\label{def:block-rank-one}
A projection $p\in\Cu(Y)$ is \emph{block-rank-one} if there are
pairwise disjoint finite nonempty sets $Y_n\subseteq Y$, $n\geq1$,
such that
\[
 p=\sum_{n\geq1}\chi_{Y_n}p\chi_{Y_n},
 \qquad \operatorname{rank}(\chi_{Y_n}p\chi_{Y_n})=1
 \quad(n\geq1).
\]
\end{defin}

The blocks need not cover $Y$ or be sparse. Every block-rank-one
projection has infinite rank and is noncompact.

\section{Graphs with expanding fibers}\label{sec:bundles}

We refer to Diestel~\cite[Chapters~1 and~2]{Die25} for standard
graph-theoretic terminology. Throughout this section, graphs are
finite, simple, and undirected, and every edge has length one.
For a graph $H$, we denote its vertex and edge sets by $V(H)$ and
$E(H)$, respectively, and write $|H|=|V(H)|$ for its order.
The girth $\girth(H)$ is the length of a shortest cycle, with
$\girth(H)=+\infty$ if $H$ is a forest. We equip each connected
component with its shortest-path metric and set the distance
between vertices in distinct components equal to $+\infty$.

Let $A_H$ denote the adjacency operator on $\ell^2(V(H))$, given by
$(A_H\xi)(v)=\sum_{u:\{u,v\}\in E(H)}\xi(u)$.
If $H$ is $d$-regular with $d\geq1$, its normalized Laplacian is
$\Delta_H=I-d^{-1}A_H$, where $I$ is the identity operator on
$\ell^2(V(H))$. We write $\sigma(T)$ for the spectrum of an operator
$T$.

We use the following two results on random regular graphs to choose
fibers with arbitrarily large girth and a uniform spectral gap.

We use the following consequence of the theorem of
Friedman~\cite[Corollary~1.4 and the subsequent discussion]{Fri08}.
\begin{lem}[\cite{Fri08}]
\label{lem:friedman}
Fix $\varepsilon>0$. Let $G_M$ be uniformly distributed among
the simple $4$-regular graphs on $M$ labeled vertices, where $M\geq6$ is even.
If $4=\lambda_1(G_M)\geq\cdots\geq\lambda_M(G_M)$ are its adjacency
eigenvalues, then
\[
\mathbb P\!\left(\max_{2\leq i\leq M}|\lambda_i(G_M)|
\leq2\sqrt{3}+\varepsilon\right)\longrightarrow1
\]
as $M\to\infty$ through admissible values.
\end{lem}

\begin{lem}[{\cite[Corollary~1]{MWW04}}]
\label{lem:random-regular-high-girth}
Fix an integer $g\geq3$, and let $G_M$ have the distribution in
Lemma~\ref{lem:friedman}. Then
\[
\mathbb P\bigl(\girth(G_M)>g\bigr)
\longrightarrow\exp\!\left(-\sum_{j=3}^{g}\frac{3^j}{2j}\right)>0
\] as $M\to\infty$ through admissible values.
\end{lem}

Lemma~\ref{lem:friedman} gives the required eigenvalue bound
with probability tending to one, whereas
Lemma~\ref{lem:random-regular-high-girth} gives a positive
limiting probability for any fixed lower bound on the girth.
The two properties therefore hold simultaneously for graphs
of arbitrarily large order.
This yields the following existence result.

\begin{lem}\label{lem:expanders-high-girth}
There are an integer $d\geq3$ and a number $\gamma>0$ such that, for all integers
$g,N\geq1$, there is a connected $d$-regular graph $H$ with
\[
|H|\geq N,\qquad \girth(H)>g,\qquad
\sigma(\Delta_H)\subseteq\{0\}\cup[\gamma,2].
\]
\end{lem}

\begin{proof}
Fix $d=4$ and choose $0<\varepsilon<d-2\sqrt{d-1}$. Set
$\gamma=1-(2\sqrt{d-1}+\varepsilon)/d>0$.
These constants are independent of $g$ and $N$.

Let $g,N\geq1$ be given. For each even $M\geq6$, let $G_M$ be
the random graph in Lemma~\ref{lem:friedman}, with adjacency
eigenvalues $d=\lambda_1(G_M)\geq\cdots\geq\lambda_M(G_M)$.
By Lemma~\ref{lem:random-regular-high-girth}, there is a number
$p_g>0$ such that
$\mathbb P(\girth(G_M)>\max\{g,3\})\geq p_g$
for all sufficiently large even $M$.
Lemma~\ref{lem:friedman} also shows that, for all sufficiently
large even $M$, the inequality
\begin{equation}\label{eq:expander-adjacency-bound}
 \max_{2\leq i\leq M}|\lambda_i(G_M)|
 \leq 2\sqrt{d-1}+\varepsilon
\end{equation}
holds with probability at least $1-p_g/2$.
The probability that the girth condition holds but
\eqref{eq:expander-adjacency-bound} fails is at most \(p_g/2\).
Thus both conditions hold with probability at least
\(p_g-p_g/2=p_g/2>0\).
Choose a sufficiently large even \(M\geq N\) and a graph \(H\)
for which both conditions hold.

The graph $H$ is connected. Indeed, if it had at least two connected
components, their indicator functions would be linearly independent
eigenvectors of $A_H$ with eigenvalue $d$. This would give
$\lambda_2(H)=d$, contradicting
\eqref{eq:expander-adjacency-bound} and
$2\sqrt{d-1}+\varepsilon<d$.

Finally, the eigenvalues of $\Delta_H$ are
$1-\lambda_i(H)/d$, for $1\leq i\leq M$.
The first is zero, while
\eqref{eq:expander-adjacency-bound} gives
$\gamma\leq1-\lambda_i(H)/d\leq2-\gamma<2$
for $2\leq i\leq M$.
Thus $\sigma(\Delta_H)\subseteq\{0\}\cup[\gamma,2]$,
as required.
\end{proof}

The next lemma allows us to join the fibers by perfect matchings
while preserving the required lower bound on the girth. Its proof
uses a switching argument.

\begin{lem}\label{lem:matching-completion}
For every pair of integers $\Delta,g\geq1$, there is an integer
$N_0=N_0(\Delta,g)$ with the following property.
Let $G$ be a graph with maximum degree at most $\Delta$ and girth
greater than $g$. Let $A,B\subseteq V(G)$ be disjoint sets with
$|A|=|B|=N\geq N_0$, and suppose that no edge of $G$ joins $A$ to $B$.
Then there is a perfect matching $M$ between $A$ and $B$ such that
the graph $G\cup M$, obtained by adding the edges of $M$ to $G$,
has girth greater than $g$.
\end{lem}

\begin{proof}
Put $D=\Delta+1$, $b=\sum_{j=0}^{g-1}D^j$, and $N_0=3b+1$.

\emph{Step 1: finding an edge to replace.}
Choose a partial matching $M$ between $A$ and $B$ of maximum
cardinality subject to $\girth(G\cup M)>g$.
Such a choice is possible because the empty matching satisfies
this condition. Write $H=G\cup M$.
Since a matching adds at most one edge at each vertex, $H$ has
maximum degree at most $D$.

For the graph metric $d_H$, write
$B_H(v,r)=\{w\in V(H):d_H(v,w)\leq r\}$.
Counting walks of lengths $0,\ldots,g-1$ starting at $v$ gives
$|B_H(v,g-1)|\leq b$.

Suppose that $M$ is not perfect, and choose unmatched vertices
$a\in A$ and $b_0\in B$.
Every unmatched vertex of $B$ lies in $B_H(a,g-1)$.
Indeed, if an unmatched vertex $v\in B$ lay outside this ball,
adding the edge $av$ would enlarge the matching without creating
a cycle of length at most $g$, contrary to the choice of $M$.
There are therefore at most $b$ unmatched vertices in $B$, and hence
$|M|\geq N-b\geq2b+1$.

The set $B_H(a,g-1)\cup B_H(b_0,g-1)$ has at most $2b$ vertices.
Since the edges of $M$ are pairwise vertex-disjoint, at most $2b$
of them meet this set. We can therefore choose an edge $a'b'\in M$,
with $a'\in A$ and $b'\in B$, such that
\begin{equation}\label{eq:matching-separation}
 d_H(u,v)\geq g
 \qquad
 \bigl(u\in\{a,b_0\},\ v\in\{a',b'\}\bigr).
\end{equation}

\emph{Step 2: enlarging the matching.}
Replace $a'b'$ by $ab'$ and $a'b_0$, setting
$M'=(M\setminus\{a'b'\})\cup\{ab',a'b_0\}$.
Because $a$ and $b_0$ are unmatched, $M'$ is a partial matching
with $|M'|=|M|+1$.
Neither added edge belongs to $G$, since $G$ has no edges between
$A$ and $B$.
The resulting graph $G\cup M'$ is therefore simple.
Figure~\ref{fig:switch} illustrates the replacement.

\begin{figure}
\centering
\begin{tikzpicture}
\tikzset{vertex/.style={circle,fill,inner sep=1.6pt}}

\node[vertex,label=left:$a$] (a) at (0,1) {};
\node[vertex,label=right:$b_0$] (b) at (2,1) {};
\node[vertex,label=left:$a'$] (c) at (0,0) {};
\node[vertex,label=right:$b'$] (d) at (2,0) {};
\draw[thick] (c)--(d);
\node at (1,-.65) {$M$};

\draw[-{Stealth},thick] (2.8,.5)--(3.8,.5);

\node[vertex,label=left:$a$] (aa) at (4.6,1) {};
\node[vertex,label=right:$b_0$] (bb) at (6.6,1) {};
\node[vertex,label=left:$a'$] (cc) at (4.6,0) {};
\node[vertex,label=right:$b'$] (dd) at (6.6,0) {};
\draw[thick] (aa)--(dd) (cc)--(bb);
\node at (5.6,-.65) {$M'$};
\end{tikzpicture}
\caption{The switch from $M$ to $M'$.
Only the affected vertices and matching edges are shown.
The vertices $a,b_0$ are initially unmatched.
The crossing on the right is not a vertex.}
\label{fig:switch}
\end{figure}

\emph{Step 3: excluding short cycles.}
Suppose that $G\cup M'$ contains a simple cycle $C$ of length at
most $g$. Since $H$ has girth greater than $g$, the cycle $C$ must
use at least one of the two new edges $ab'$ and $a'b_0$.

If $C$ uses exactly one new edge, deleting that edge leaves a path
in $H$ of length at most $g-1$ between its endpoints.
This contradicts \eqref{eq:matching-separation}.

Suppose that $C$ uses both new edges.
Deleting them leaves two paths in $H$, neither of which uses the
removed edge $a'b'$.
Their endpoint pairs are either $(a,a')$ and $(b_0,b')$, or
$(a,b_0)$ and $(a',b')$.
In the first case, \eqref{eq:matching-separation} implies that
each path has length at least $g$, so $C$ has length at least
$2g+2$, a contradiction.

In the second case, let $P$ be the path from $a'$ to $b'$.
Adding the edge $a'b'$ to $P$ gives a simple cycle in $H$.
Since $\girth(H)>g$, the path $P$ has length at least $g$.
Together with the two new edges, this already gives at least
$g+2$ edges of $C$, again a contradiction.

Thus $G\cup M'$ has girth greater than $g$.
Since $|M'|=|M|+1$, this contradicts the choice of $M$ with maximum
cardinality. Hence $M$ is perfect.
\end{proof}

For $L\geq3$, let $C_L$ be the cycle with vertices
$x_i=i+L\mathbb Z$, $0\leq i<L$, and edges
$\{x_i,x_{i+1}\}$, where the indices are taken modulo $L$.

\begin{prop}\label{prop:bundle}
There are an integer $d\geq3$ and a number $\gamma>0$ such that,
for all integers $L\geq3$, $g\geq1$, and $M_0\geq1$, there exist
a connected $(d+2)$-regular graph $Y$ of girth greater than $g$
and a surjection $\pi\colon V(Y)\to V(C_L)$ with the following
properties.
\begin{enumerate}
\item
Each fiber $F_x=\pi^{-1}(x)$ induces a graph isomorphic to the
same connected $d$-regular graph $H$, where $|H|\geq M_0$ and
$\sigma(\Delta_H)\subseteq\{0\}\cup[\gamma,2]$.

\item
For distinct $x,z\in V(C_L)$, the edges joining $F_x$ and $F_z$
form a perfect matching if $x$ and $z$ are adjacent, and there
are no such edges otherwise.
\end{enumerate}
\end{prop}

\begin{proof}
Choose $d,\gamma$ as in Lemma~\ref{lem:expanders-high-girth},
and let $L,g,M_0$ be given.

\emph{Step 1: choosing the fibers.}
Let $N_0=N_0(d+2,g)$ be the threshold in
Lemma~\ref{lem:matching-completion}.
By Lemma~\ref{lem:expanders-high-girth}, there is a connected
$d$-regular graph $H$ satisfying
\begin{equation}\label{eq:bundle-fiber-choice}
 |H|\geq\max\{M_0,N_0\},
 \qquad \girth(H)>g,
 \qquad \sigma(\Delta_H)\subseteq\{0\}\cup[\gamma,2].
\end{equation}
For each $x\in V(C_L)$, take a copy of $H$ with vertex set $F_x$,
with these copies pairwise disjoint.
Let $G_0$ be their disjoint union, and define $\pi(v)=x$ for
$v\in F_x$.

\emph{Step 2: joining adjacent fibers.}
For $j=1,\ldots,L$, we construct $G_j$ from $G_{j-1}$ by adding
a perfect matching between $F_{x_{j-1}}$ and $F_{x_j}$,
where $x_L=x_0$.
We maintain the conditions that $G_j$ has maximum degree at most
$d+2$ and girth greater than $g$.
Both hold for $G_0$ by \eqref{eq:bundle-fiber-choice}.

Suppose that $G_{j-1}$ has been constructed.
No edge of $G_{j-1}$ joins $F_{x_{j-1}}$ to $F_{x_j}$, since
the base edge $\{x_{j-1},x_j\}$ has not yet been processed.
The two fibers have the same cardinality $|H|\geq N_0$.
Lemma~\ref{lem:matching-completion}, applied with
$\Delta=d+2$, therefore supplies a perfect matching whose
addition preserves girth greater than $g$.

Each of the two base vertices $x_{j-1},x_j$ is incident to only
one base edge other than the edge currently being processed.
Thus every vertex in these two fibers has at most one neighbor
outside its fiber before the new matching is added.
Its degree after the addition is consequently at most $d+2$.
The degrees of all other vertices remain unchanged.
This completes the induction.

Set $Y=G_L$. Then $\girth(Y)>g$, and $\pi$ is defined on $V(Y)$
because no vertices were added or removed.
Figure~\ref{fig:bundle} shows the matchings over two consecutive
base edges.

\begin{figure}
\centering
\begin{tikzpicture}
\tikzset{vertex/.style={circle,fill,inner sep=1.6pt}}

\foreach \j/\x/\lab in
  {0/0/{x_{i-1}},1/3.2/{x_i},2/6.4/{x_{i+1}}}{
  \draw[rounded corners=7pt,fill=black!3]
    ({\x-.6},.25) rectangle ({\x+.6},2.65);
  \node at (\x,3) {$F_{\lab}$};

  \foreach \k/\yy in {0/.6,1/1.15,2/1.7,3/2.25}{
    \node[vertex] (v\j\k) at (\x,\yy) {};
  }

  \node[vertex,label=below:$\lab$] (b\j) at (\x,-1) {};
  \draw[-{Stealth}]
    (\x,.15)--node[right] {$\pi$}(\x,-.75);
}

\draw[black!45]
  (v00)--(v12)
  (v02)--(v13)
  (v03)--(v11)
  (v11)--(v23)
  (v12)--(v20)
  (v13)--(v22);

\draw[very thick] (v01)--(v10)--(v21);
\draw[very thick] (b0)--(b1)--(b2);

\draw[densely dotted]
  (-1.1,1.45)--(-.65,1.45)
  (7.05,1.45)--(7.5,1.45)
  (-1.1,-1)--(-.15,-1)
  (6.55,-1)--(7.5,-1);

\node at (-1.6,1.45) {$Y$};
\node at (-1.6,-1) {$C_L$};
\node at (1.6,2.85) {\small matching};
\node at (4.8,2.85) {\small matching};
\end{tikzpicture}
\caption{The matchings over two consecutive edges of $C_L$.
Each fiber induces a copy of $H$.
Only a few vertices and matching edges are shown; all edges
within the fibers are omitted.
The bold path in $Y$ lifts the bold path in $C_L$.}
\label{fig:bundle}
\end{figure}

\emph{Step 3: regularity and connectedness.}
Every vertex of $Y$ has $d$ neighbors within its fiber and one
neighbor for each of the two base edges incident to that fiber.
Hence $Y$ is $(d+2)$-regular.
All added edges join distinct fibers, so each fiber still induces
a copy of $H$. The first assertion follows from
\eqref{eq:bundle-fiber-choice}, and the second holds by construction.
Every fiber is nonempty, so $\pi$ is surjective.

Finally, let $u\in F_x$ and $v\in F_z$.
Choose a path from $x$ to $z$ in $C_L$ and, starting at $u$,
follow the matchings along its edges.
This gives a path in $Y$ ending at some vertex of $F_z$.
Since $F_z$ induces a connected graph, this path can be continued
within $F_z$ to $v$. Thus $Y$ is connected.
\end{proof}

The map $\pi$ in Proposition~\ref{prop:bundle} is $1$-Lipschitz,
since the endpoints of every edge of $Y$ have equal or adjacent
images in $C_L$. Each oriented path from $x$ to $z$ in $C_L$
also determines a bijection $F_x\to F_z$, obtained by composing
the bijections given by the perfect matchings along the path.
We will use the Lipschitz bound and these bijections in
Section~\ref{sec:corner}.

\section{Nonexistence of coarse embeddings}\label{sec:obstruction}

Fix $d$ and $\gamma$ as in Proposition~\ref{prop:bundle}.
Choose an integer $L_1\geq3$.
Once $L_n$ has been chosen, apply Proposition~\ref{prop:bundle}
with parameters $L=L_n$, $g=nL_n$, and $M_0=n$.
This gives a connected $(d+2)$-regular graph $Y_n$ and a surjection
$\pi_n\colon V(Y_n)\to V(C_{L_n})$, whose fibers all have the
same cardinality $M_n$, with
\begin{equation}\label{eq:bundle-parameters}
 \girth(Y_n)>nL_n,
 \qquad M_n\geq n.
\end{equation}
After constructing $Y_n$, choose an integer $L_{n+1}$ such that
\begin{equation}\label{eq:length-recursion}
 L_{n+1}>(n+1)\sum_{j=1}^{n}|Y_j|.
\end{equation}
Since $|Y_n|=L_nM_n\geq L_n$,
\eqref{eq:length-recursion} implies $L_{n+1}>L_n$.
Thus $(L_n)$ is a strictly increasing sequence of integers and
tends to infinity.

Regard each graph as its vertex set equipped with the graph metric,
and set
\begin{equation}\label{eq:X-and-Y}
 X=\bigsqcup_{n\geq1}C_{L_n},
 \qquad
 Y=\bigsqcup_{n\geq1}Y_n.
\end{equation}
Equip $X$ and $Y$ with coarse disjoint union metrics as in
Definition~\ref{def:coarse-disjoint-union}.
Both spaces have bounded geometry.
Indeed, all component degrees are at most $d+2$, so a ball of
radius $r$ contained in one component has at most
$\sum_{j=0}^{\lfloor r\rfloor}(d+2)^j$ vertices.
For each fixed $r>0$,
Definition~\ref{def:coarse-disjoint-union} ensures that only
finitely many components lie within distance $r$ of another
component. Their union is finite and contains every ball of
radius $r$ that meets more than one component.
This gives a uniform bound on the cardinalities of all balls
of radius $r$.

The next lemma uses the uniqueness of paths in a tree;
see Diestel~\cite[Section~1.5]{Die25}.
We explain explicitly how the three relevant geodesics meet.

\begin{lem}\label{lem:tree-median}
Let $L\geq3$ be an integer, let $R\geq0$, and let
$x_0,\ldots,x_{L-1}$ be the cyclically ordered vertices of $C_L$.
Suppose that $z_0,\ldots,z_L$ are vertices of a tree $T$ such that
$z_L=z_0$ and $d_T(z_i,z_{i+1})\leq R$ for $0\leq i<L$.
Then there are $r,s\in\{0,\ldots,L-1\}$ such that
\[
 d_{C_L}(x_r,x_s)\geq
 \frac12\left\lfloor\frac L3\right\rfloor,
 \qquad
 d_T(z_r,z_s)\leq2R.
\]
\end{lem}

\begin{proof}
Set $k=\lfloor L/3\rfloor$ and $x_L=x_0$.
We retain indices in $\{0,\ldots,L\}$ during the proof and reduce
the chosen indices modulo $L$ at the end.

\emph{Step 1: choosing three images near a common vertex.}
For vertices \(u,v\in T\), write \([u,v]\) for the unique
geodesic joining them.
The intersection of \([z_0,z_k]\) and \([z_0,z_{2k}]\)
is an initial segment \([z_0,c]\).
The remaining segments \([c,z_k]\) and \([c,z_{2k}]\)
meet only at \(c\), so their union is the geodesic
\([z_k,z_{2k}]\).
Thus \(c\) is the unique vertex common to the three pairwise
geodesics. This vertex is called the median of
\(z_0,z_k,z_{2k}\).

Join each \(z_i\) to \(z_{i+1}\) by its geodesic in \(T\).
Concatenating these geodesics over the index intervals
\([0,k]\), \([k,2k]\), and \([2k,L]\) gives three walks.
Erasing closed subwalks from any walk leaves the unique
geodesic between its endpoints.
Each of the three walks therefore contains \(c\).

On each of these walks, choose a constituent geodesic segment
containing $c$.
Its length is at most $R$, so either endpoint is within distance
$R$ of $c$.
We can therefore choose integers
$r_1\in[0,k]$, $r_2\in[k,2k]$, and $r_3\in[2k,L]$ such that
\begin{equation}\label{eq:median-image-bounds}
 d_T(z_{r_i},c)\leq R
 \qquad(i=1,2,3).
\end{equation}
The three chosen images lie in the closed ball $B_T(c,R)$
of radius $R$ centered at $c$; see Figure~\ref{fig:median}.

\begin{figure}
\centering
\begin{tikzpicture}
\tikzset{
  vertex/.style={circle,fill,inner sep=1.6pt},
  selected/.style={circle,draw,fill=white,inner sep=2pt}
}

\draw (0,0) circle (1.45);
\node[vertex,label=above:$x_0$] at (90:1.45) {};
\node[vertex,label=below left:$x_k$] at (210:1.45) {};
\node[vertex,label=below right:$x_{2k}$] at (330:1.45) {};
\node[selected,label=above left:$x_{r_1}$] at (150:1.45) {};
\node[selected,label=below:$x_{r_2}$] at (270:1.45) {};
\node[selected,label=above right:$x_{r_3}$] at (30:1.45) {};
\node at (0,0) {$C_L$};

\draw[-{Stealth}]
  (2.2,0)--node[above] {$x_i\mapsto z_i$}(4.3,0);

\begin{scope}[shift={(6.7,0)}]
  \draw[draw=none,fill=black!8] (0,0) circle (.85);
  \draw[thick]
    (0,1.75)--(0,0)--(-1.65,-1.1)
    (0,0)--(1.65,-1.1);

  \node[vertex,label=above:$z_0$] at (0,1.75) {};
  \node[vertex,label=below left:$z_k$] at (-1.65,-1.1) {};
  \node[vertex,label=below right:$z_{2k}$] at (1.65,-1.1) {};
  \node[vertex,label=above left:$c$] at (0,0) {};

  \node[selected,label=right:$z_{r_3}$] at (0,.6) {};
  \node[selected,label=above right:$z_{r_2}$] at (.6,-.4) {};
  \node[selected,label=above left:$z_{r_1}$] at (-.6,-.4) {};
  \node at (0,-1.85) {$B_T(c,R)$};
\end{scope}
\end{tikzpicture}
\caption{The three arcs of $C_L$ give walks in $T$ that all
pass through the median $c$.
Choosing an endpoint of a segment through $c$ on each walk
gives the three images in $B_T(c,R)$.
The drawing is schematic; some selected vertices may coincide.}
\label{fig:median}
\end{figure}

\emph{Step 2: finding two vertices far apart on the cycle.}
For integers $0\leq p\leq q\leq L$, the cycle metric satisfies
\begin{equation}\label{eq:median-cycle-distance}
 d_{C_L}(x_p,x_q)
 =\min\{q-p,L-(q-p)\}.
\end{equation}
If $d_{C_L}(x_{r_1},x_{r_2})\geq k/2$, choose this pair.

Otherwise, since
$L-(r_2-r_1)\geq L-2k\geq k$,
\eqref{eq:median-cycle-distance} forces $r_2-r_1<k/2$.
Together with $r_2\geq k$, this gives $r_1>k/2$.
On the other hand, $r_3-r_1\geq k$ and
$L-(r_3-r_1)\geq r_1>k/2$.
Applying \eqref{eq:median-cycle-distance} to $r_1,r_3$ therefore
gives $d_{C_L}(x_{r_1},x_{r_3})>k/2$.

Reduce the chosen pair of indices modulo $L$, and denote the
results by $r,s\in\{0,\ldots,L-1\}$.
Their cyclic distance is unchanged.
Since $z_L=z_0$, \eqref{eq:median-image-bounds} and the triangle
inequality give
$d_T(z_r,z_s)\leq d_T(z_r,c)+d_T(c,z_s)\leq2R$,
as required.
\end{proof}

\begin{prop}\label{prop:no-embedding}
There is no coarse embedding from $X$ into $Y$.
\end{prop}

\begin{proof}
Suppose that $f\colon X\to Y$ is a coarse embedding.

\emph{Step 1: locating the image of a sufficiently long cycle.}
Choose an integer $R\geq1$ and a number $S\geq0$ such that,
for all $x,x'\in X$,
\begin{align}
 d_X(x,x')=1
 &\ \Longrightarrow\
 d_Y(f(x),f(x'))\leq R,
 \label{eq:edge-control}\\
 d_Y(f(x),f(x'))\leq2R
 &\ \Longrightarrow\
 d_X(x,x')\leq S.
 \label{eq:proper-control}
\end{align}
Set $N_f=\sup_{x\in X}|B_X(x,S)|<\infty$.
By \eqref{eq:proper-control}, every nonempty fiber of $f$ is
contained in an $S$-ball in $X$, and hence has cardinality at
most $N_f$.

Choose $n_0$ so that every $Y_m$ with $m\geq n_0$ is at distance
greater than $R$ from all other components, and put
$D_0=\bigsqcup_{m<n_0}Y_m$.
Since $L_n\to\infty$, we can choose an integer $n$ such that
\begin{equation}\label{eq:choose-n}
 n>\max\{n_0,R,N_f\},
 \qquad
 \frac12\left\lfloor\frac{L_n}{3}\right\rfloor>S.
\end{equation}
The recursion \eqref{eq:length-recursion}, together with $n>N_f$,
gives
\begin{equation}\label{eq:early-component-count}
 L_n>N_f\sum_{j<n}|Y_j|.
\end{equation}

The image $f(C_{L_n})$ cannot be contained in $D_0$.
Indeed, such an inclusion would give $L_n\leq N_f|D_0|$,
contrary to \eqref{eq:early-component-count}, since $n>n_0$.
Thus $f(C_{L_n})$ meets some $Y_m$ with $m\geq n_0$.
Successive vertices of the cycle have images at distance at most
$R$ by \eqref{eq:edge-control}, whereas $Y_m$ is at distance
greater than $R$ from every other component.
Following the cycle from a vertex whose image lies in $Y_m$
therefore shows that $f(C_{L_n})\subseteq Y_m$.

We must also have $m\geq n$.
Otherwise, counting the preimages of points in $Y_m$ would give
$L_n\leq N_f|Y_m|\leq N_f\sum_{j<n}|Y_j|$,
again contradicting \eqref{eq:early-component-count}.

\emph{Step 2: placing the image in a tree.}
Write $x_0,\ldots,x_{L_n-1}$ for the cyclically ordered vertices
of $C_{L_n}$, set $x_{L_n}=x_0$, and put $z_i=f(x_i)$.
The metric of $Y$ restricts to the graph metric on $Y_m$.
Thus \eqref{eq:edge-control} allows us to join each $z_i$ to
$z_{i+1}$ by a geodesic in $Y_m$ of length at most $R$.
Let $K$ be the union of these geodesics.
Then $K$ is connected and has at most $RL_n$ edges.

Since $m\geq n$ and $L_m\geq L_n$,
\eqref{eq:bundle-parameters} and \eqref{eq:choose-n} give
\begin{equation}\label{eq:tree-girth-bound}
 \girth(Y_m)>mL_m\geq nL_n>RL_n\geq|E(K)|.
\end{equation}
Any cycle in $K$ would be a cycle in $Y_m$ of length at most
$|E(K)|$, contradicting \eqref{eq:tree-girth-bound}.
Hence $K$ is a tree.

Let $d_K$ denote its graph metric.
The chosen geodesics show that $d_K(z_i,z_{i+1})\leq R$,
and $d_Y\leq d_K$ on $V(K)$.
Lemma~\ref{lem:tree-median} and \eqref{eq:choose-n} therefore
give indices $r,s$ such that
\begin{align*}
 d_X(x_r,x_s)
 &=d_{C_{L_n}}(x_r,x_s)
 \geq\frac12\left\lfloor\frac{L_n}{3}\right\rfloor>S,\\
 d_Y(f(x_r),f(x_s))
 &\leq d_K(z_r,z_s)\leq2R.
\end{align*}
This contradicts \eqref{eq:proper-control}.
\end{proof}

The next proposition will also rule out a coarse embedding
from \(Y\) into \(X\).
For the relation between a chord and the radius of its circle,
we refer to Coxeter and Greitzer~\cite[Theorem~1.11]{CG67}.
The spectral estimate used below is the variational formula
for the normalized graph Laplacian; see
Chung~\cite[Section~1.2]{Chu97}.
We include the relevant calculations in the proof.

\begin{prop}\label{prop:additional-geometry}
The space $X$ has asymptotic dimension one and admits a
bi-Lipschitz embedding into a Hilbert space.
The space $Y$ admits no coarse embedding into any Hilbert space.
\end{prop}

\begin{proof}
\emph{Step 1: the asymptotic dimension of $X$.}
Fix $r>0$ and choose an integer $a>r$.
By Definition~\ref{def:coarse-disjoint-union} and $L_n\to\infty$,
there is $n_1$ such that, for every $n\geq n_1$,
the cycle $C_{L_n}$ has length at least $4a$ and is at distance
greater than $r$ from every other component.

For such a cycle of length $L$, put
$k=2\lfloor L/(4a)\rfloor\geq2$.
Partition its vertices into $k$ consecutive arcs, each containing
either $\lfloor L/k\rfloor$ or $\lceil L/k\rceil$ vertices.
Since $2a\leq L/k<4a$, every arc contains between $2a$ and $4a$
vertices and has diameter at most $4a-1$.

Color the arcs alternately with two colors.
Because $k$ is even, adjacent arcs have different colors also
where the partition closes around the cycle.
Any two distinct arcs of the same color are separated, in either
direction around the cycle, by an arc containing at least $2a$
vertices. Their distance is therefore at least $2a+1>r$.

Collect the arcs of each color over all $n\geq n_1$ into one
family, and add the finite union
$\bigsqcup_{n<n_1}C_{L_n}$ as a single set to the first family.
Each family is $r$-disjoint, and together they form a uniformly
bounded cover of $X$.
This proves $\operatorname{asdim}X\leq1$.

If $\operatorname{asdim}X=0$, there would be a uniformly bounded
cover of $X$ consisting of a single $1$-disjoint family.
Adjacent vertices of any $C_{L_n}$ would have to lie in the same
member of this cover. Connectedness would then place the whole
cycle in one member.
This is impossible because
$\operatorname{diam}(C_{L_n})=\lfloor L_n/2\rfloor\to\infty$.
Thus $\operatorname{asdim}X=1$.

\emph{Step 2: a bi-Lipschitz embedding of $X$ into Hilbert space.}
Let $o_n\in C_{L_n}$ and $(t_n)$ be the basepoints and sequence
used to define the metric on $X$ in
Definition~\ref{def:coarse-disjoint-union}.
For each $n$, let $E_n$ be a Euclidean plane.
Map the vertices of $C_{L_n}$, in cyclic order, to the vertices
of a regular polygon in $E_n$ of circumradius $L_n/(2\pi)$,
and translate the polygon so that the resulting map
$v_n\colon C_{L_n}\to E_n$ satisfies $v_n(o_n)=0$.

For \(x,x'\in C_{L_n}\), put
\(\ell=d_{C_{L_n}}(x,x')\leq L_n/2\).
The radius vectors of the corresponding polygon vertices
form an angle \(2\pi\ell/L_n\).
Writing these vertices in polar coordinates gives
\begin{equation}\label{eq:polygon-chord}
\begin{aligned}
 \|v_n(x)-v_n(x')\|^2
 &=2\left(\frac{L_n}{2\pi}\right)^2
   \left(1-\cos\frac{2\pi\ell}{L_n}\right)\\
 &=\left(\frac{L_n}{\pi}\right)^2
   \sin^2\frac{\pi\ell}{L_n}.
\end{aligned}
\end{equation}
Translation of the polygon does not affect this calculation.

On \([0,\pi/2]\), the concavity of sine gives
\(\sin u\geq 2u/\pi\), by comparison with the line joining
\((0,0)\) and \((\pi/2,1)\).
Also,
\(\sin u=\int_0^u\cos t\,dt\leq u\).
Taking square roots in \eqref{eq:polygon-chord} and applying
these inequalities with \(u=\pi\ell/L_n\), we obtain
\begin{equation}\label{eq:polygon-distance}
 \frac2\pi d_{C_{L_n}}(x,x')
 \leq \|v_n(x)-v_n(x')\|
 \leq d_{C_{L_n}}(x,x').
\end{equation}

Let $\mathcal K=\mathbb R\oplus\bigoplus_{n\geq1}E_n$ be the
Hilbert direct sum, and let $e_0$ be the unit vector in the
first summand.
Define $F\colon X\to\mathcal K$ by
$F(x)=t_ne_0+v_n(x)$ for $x\in C_{L_n}$.

Suppose that $x\in C_{L_n}$ and $x'\in C_{L_m}$ with $n\ne m$,
and write
$a=d_X(x,o_n)$, $b=d_X(x',o_m)$, and $c=|t_n-t_m|$.
Definition~\ref{def:coarse-disjoint-union} gives
$d_X(x,x')=a+b+c$, while orthogonality gives
\begin{equation}\label{eq:polygon-cross-component-norm}
 \|F(x)-F(x')\|^2
 =c^2+\|v_n(x)\|^2+\|v_m(x')\|^2.
\end{equation}
By \eqref{eq:polygon-distance}, the right-hand side of
\eqref{eq:polygon-cross-component-norm} lies between
$(4/\pi^2)(a^2+b^2+c^2)$ and $a^2+b^2+c^2$.
The inequalities
$a^2+b^2+c^2\leq(a+b+c)^2\leq3(a^2+b^2+c^2)$
therefore yield
\begin{equation}\label{eq:hilbert-embedding-bounds}
 \frac{2}{\pi\sqrt3}d_X(x,x')
 \leq\|F(x)-F(x')\|
 \leq d_X(x,x').
\end{equation}
For points in the same component,
\eqref{eq:hilbert-embedding-bounds} follows directly from
\eqref{eq:polygon-distance}.
Hence $F$ is a bi-Lipschitz embedding.

\emph{Step 3: excluding coarse embeddings of $Y$ into Hilbert space.}
Suppose that $h\colon Y\to\mathcal H$ is a coarse embedding with
control functions $\rho_-,\rho_+$.
For each $n$, choose a fiber $F_n$ of $\pi_n$, and let $H_n$ be
the subgraph of $Y_n$ induced by $F_n$.
By Proposition~\ref{prop:bundle}, the graph $H_n$ is connected
and $d$-regular, has $M_n$ vertices, and satisfies
$\sigma(\Delta_{H_n})\subseteq\{0\}\cup[\gamma,2]$.
Put $\overline h_n=M_n^{-1}\sum_{u\in F_n}h(u)$.

Expanding the squared distances around \(\overline h_n\)
and applying the spectral gap to each scalar coordinate gives
\begin{equation}\label{eq:fiber-poincare}
\begin{aligned}
 \frac1{M_n^2}\sum_{u,v\in F_n}\|h(u)-h(v)\|^2
 &=\frac2{M_n}\sum_{u\in F_n}\|h(u)-\overline h_n\|^2\\
 &\leq\frac{2}{\gamma dM_n}
 \sum_{\{u,v\}\in E(H_n)}\|h(u)-h(v)\|^2\\
 &\leq\frac{\rho_+(1)^2}{\gamma}.
\end{aligned}
\end{equation}
The equality follows by expanding the squared norms and using
$\sum_{u\in F_n}(h(u)-\overline h_n)=0$.
For the first inequality, choose an orthonormal basis of the
finite-dimensional span of $h(F_n)$.
Each scalar coordinate $\varphi$ of
$h|_{F_n}-\overline h_n$ has mean zero, so the spectral gap gives
$\gamma\|\varphi\|_2^2
\leq\langle\Delta_{H_n}\varphi,\varphi\rangle$.
The quadratic-form identity
$\langle\Delta_{H_n}\varphi,\varphi\rangle
=d^{-1}\sum_{\{u,v\}\in E(H_n)}
|\varphi(u)-\varphi(v)|^2$,
summed over the coordinates, proves the first inequality in
\eqref{eq:fiber-poincare}.
For the last inequality, use
$|E(H_n)|=dM_n/2$ and the fact that the endpoints of each edge
of $H_n$ have distance one in $Y$.

Now fix $s>0$ and put
$b_s=\sup_{y\in Y}|B_Y(y,s)|<\infty$.
For each $u\in F_n$, at most $b_s$ points $v\in F_n$ satisfy
$d_Y(u,v)\leq s$.
Thus at most $M_nb_s$ ordered pairs in $F_n\times F_n$ have
$Y$-distance at most $s$.
The lower control bound and the monotonicity of $\rho_-$ give
\begin{equation}\label{eq:fiber-pair-lower-bound}
 \frac1{M_n^2}\sum_{u,v\in F_n}\|h(u)-h(v)\|^2
 \geq
 \left(1-\frac{b_s}{M_n}\right)\rho_-(s)^2.
\end{equation}
Combining \eqref{eq:fiber-poincare} and
\eqref{eq:fiber-pair-lower-bound}, and then letting $n\to\infty$
using $M_n\geq n$ from \eqref{eq:bundle-parameters}, we obtain
$\rho_-(s)^2\leq\rho_+(1)^2/\gamma$ for every $s>0$.
This contradicts $\rho_-(s)\to\infty$ as $s\to\infty$.
\end{proof}

Proposition~\ref{prop:additional-geometry} excludes a coarse
embedding $Y\to X$: composing such an embedding with a
bi-Lipschitz embedding of $X$ into Hilbert space would give a
coarse embedding of $Y$ into Hilbert space.
The argument in Step~3 uses the ambient metric $d_Y$ to measure
distances between points of a fiber.
It requires no assumption that the fibers are isometrically
embedded in $Y$.

\section{The hereditary corner}\label{sec:corner}

For $x\in C_{L_n}$, write $F_x=\pi_n^{-1}(x)$. Define
$U\colon\ell^2(X)\to\ell^2(Y)$ by
\begin{equation}\label{eq:isometry-U}
U\delta_x=\xi_x:=M_n^{-1/2}\sum_{y\in F_x}\delta_y,
\qquad x\in C_{L_n}.
\end{equation}
The fibers are pairwise disjoint, and each $\xi_x$ is a unit vector,
so $U$ is an isometry. Put $q=UU^*$. The range of $q$ consists of
the square-summable functions on $Y$ that are constant on each fiber.

The uniform spectral gap allows us to approximate \(q\)
in norm by operators of finite propagation.
This is the construction of projections from graph Laplacians
discussed by Roe and Willett~\cite[Section~3]{RW14}.
In the present setting, the approximating operators can be
written explicitly.

\begin{lem}\label{lem:q-in-roe}
The projection \(q\) belongs to \(\Cu(Y)\).
\end{lem}

\begin{proof}
For each \(x\in X\), let \(\Delta_{F_x}\) be the normalized
Laplacian of the graph induced by \(F_x\), and let \(q_x\)
be the orthogonal projection onto \(\mathbb C\xi_x\).
Define
\[
 P_x=I_{\ell^2(F_x)}-\frac12\Delta_{F_x},
 \qquad
 P=\bigoplus_{x\in X}P_x.
\]
Each \(P_x\) is a positive contraction.
Its nonzero matrix entries occur only on the diagonal
and along edges within \(F_x\).
Hence \(P\) is a bounded operator of propagation at most one.

Since each fiber is connected, \(P_x\) acts as the identity
on \(\mathbb C\xi_x\).
On its orthogonal complement, the spectral gap in
Proposition~\ref{prop:bundle} places the spectrum of \(P_x\)
in \([0,1-\gamma/2]\).
Diagonalizing the finite-dimensional self-adjoint operator
\(P_x\) therefore gives
\(\|P_x^m-q_x\|\leq(1-\gamma/2)^m\) for every \(m\geq1\).
Since \(q=\bigoplus_{x\in X}q_x\), we obtain
\begin{equation}\label{eq:q-polynomial-approximation}
 \|P^m-q\|
 =\sup_{x\in X}\|P_x^m-q_x\|
 \leq(1-\gamma/2)^m
 \longrightarrow0.
\end{equation}
The operator \(P^m\) has propagation at most \(m\).
Thus \eqref{eq:q-polynomial-approximation} expresses \(q\)
as a norm limit of finite-propagation operators, proving
that \(q\in\Cu(Y)\).
\end{proof}

\begin{lem}\label{lem:compression}
For every $T\in\Cu(Y)$, one has $U^*TU\in\Cu(X)$.
\end{lem}

\begin{proof}
Suppose first that $T$ has propagation at most $r$, and put
$S=U^*TU$. For $x\in C_{L_n}$ and $x'\in C_{L_m}$,
\eqref{eq:isometry-U} gives
\begin{equation}\label{eq:compression-matrix}
\langle S\delta_{x'},\delta_x\rangle
=
\frac{1}{\sqrt{M_nM_m}}
\sum_{y\in F_x}\sum_{z\in F_{x'}}
\langle T\delta_z,\delta_y\rangle.
\end{equation}

If $n=m$ and the coefficient in \eqref{eq:compression-matrix}
is nonzero, there are $y\in F_x$ and $z\in F_{x'}$ with
$\langle T\delta_z,\delta_y\rangle\ne0$. The propagation bound
and the $1$-Lipschitz property of $\pi_n$ imply
$d_X(x,x')\leq d_Y(y,z)\leq r$. Hence the block-diagonal part of
$S$, with respect to the components $C_{L_n}$, has propagation
at most $r$.

If $n\ne m$ and the coefficient in
\eqref{eq:compression-matrix} is nonzero, then
$d_Y(Y_n,Y_m)\leq r$. By
Definition~\ref{def:coarse-disjoint-union}, only finitely many
pairs of distinct components satisfy this inequality. Since
each component is finite, the off-diagonal component blocks of
$S$ have finite support. Thus $S$ has finite propagation and
belongs to $\Cu(X)$.

For general $T\in\Cu(Y)$, approximate $T$ in norm by
finite-propagation operators and use the contractivity of
$T\mapsto U^*TU$.
\end{proof}
A \emph{partial translation} of \(X\) is a bijection
\(t\colon D\to E\) between subsets of \(X\) such that
\(\sup_{x\in D}d_X(x,t(x))<\infty\).
Given a bounded function \(w\colon D\to\mathbb C\),
the associated weighted partial translation sends
\(\delta_x\) to \(w(x)\delta_{t(x)}\) for \(x\in D\)
and to zero for \(x\notin D\).
Its norm is \(\|w\|_\infty\), since distinct basis vectors
with indices in \(D\) have orthogonal images.
\begin{lem}\label{lem:lifting}
For every $a\in\Cu(X)$, one has $UaU^*\in\Cu(Y)$.
\end{lem}

\begin{proof}
We first consider an operator $a$ of propagation at most an
integer $r\geq0$.

\emph{Step 1: decomposing $a$ into partial translations.}
Let $a_0$ be the block-diagonal part of $a$ with respect to
$\ell^2(X)=\bigoplus_{n\geq1}\ell^2(C_{L_n})$, and put
$k=a-a_0$. Then $\|a_0\|\leq\|a\|$. As in the proof of
Lemma~\ref{lem:compression}, the coarse disjoint union metric
and the propagation bound imply that $k$ has finite support.
In particular, $UkU^*$ has finite rank, so
\eqref{eq:compacts-in-roe} gives $UkU^*\in\Cu(Y)$.

Give each cycle $C_{L_n}=\mathbb Z/L_n\mathbb Z$ the orientation
$x\mapsto x+1$. Write
$(a_0)_{x,x'}=\langle a_0\delta_{x'},\delta_x\rangle$.
For every nonzero entry $(a_0)_{x,x'}$, choose a shortest path
from $x'$ to $x$ in their common cycle, taking the positively
oriented path in the antipodal case. Let $j(x,x')$ be its signed
length. Thus $|j(x,x')|\leq r$ and
$x=x'+j(x,x')$ in $\mathbb Z/L_n\mathbb Z$.

For each $j\in\{-r,\ldots,r\}$, let $a^{(j)}$ retain precisely
the entries of $a_0$ for which $j(x,x')=j$, with all other
entries set equal to zero. For fixed $j$, the retained pairs
form the graph of a partial bijection: on each cycle, they
are a restriction of translation by $j$. Hence $a^{(j)}$ is a
bounded weighted partial translation, and
\begin{equation}\label{eq:cyclic-translation-decomposition}
a_0=\sum_{j=-r}^{r}a^{(j)},
\qquad
\|a^{(j)}\|\leq\|a_0\|\leq\|a\|.
\end{equation}

\emph{Step 2: lifting the partial translations.}
For a pair $(x,x')$ appearing in $a^{(j)}$, compose the
matchings along the chosen path from $x'$ to $x$. This gives
a bijection
$\theta_{x,x'}\colon F_{x'}\to F_x$.
When $j=0$, this is the identity on $F_x$.

Define $\widetilde a^{(j)}$ on the standard basis of
$\ell^2(Y)$ by
\begin{equation}\label{eq:lifted-partial-translation}
\widetilde a^{(j)}\delta_y
=
\begin{cases}
(a_0)_{x,x'}\,\delta_{\theta_{x,x'}(y)},
&
y\in F_{x'}\text{ and }(x,x')\text{ appears in }a^{(j)},\\
0,
&
\text{otherwise}.
\end{cases}
\end{equation}
For fixed \(j\), each retained pair satisfies \(x=x'+j\)
in its cycle. Since translation by \(j\) is a bijection on
each cycle, distinct vertices \(x'\) correspond to distinct
vertices \(x\). The bijections
\(\theta_{x,x'}\colon F_{x'}\to F_x\) therefore have pairwise
disjoint domains and pairwise disjoint ranges.
Thus \eqref{eq:lifted-partial-translation} defines a bounded
weighted partial translation with
\(\|\widetilde a^{(j)}\|=\|a^{(j)}\|\).

The path used to define \(\theta_{x,x'}\) has length \(|j|\).
Following the matchings along this path joins each
\(y\in F_{x'}\) to \(\theta_{x,x'}(y)\) by a path in \(Y\)
of the same length.
Every nonzero matrix coefficient in
\eqref{eq:lifted-partial-translation} therefore joins vertices
at distance at most \(|j|\), so
\(\operatorname{prop}(\widetilde a^{(j)})\leq |j|\leq r\).

The fibers $F_x$ and $F_{x'}$ have the same cardinality.
Consequently, the bijection $\theta_{x,x'}$ sends their
normalized constant vectors to one another, and
$\widetilde a^{(j)}U=Ua^{(j)}$. Multiplying on the right
by $U^*$ gives
\begin{equation}\label{eq:lifted-shift-identity}
Ua^{(j)}U^*=\widetilde a^{(j)}q.
\end{equation}

\emph{Step 3: recovering $UaU^*$.}
Equations~\eqref{eq:cyclic-translation-decomposition}
and~\eqref{eq:lifted-shift-identity} yield
\[
UaU^*
=
UkU^*+\sum_{j=-r}^{r}\widetilde a^{(j)}q.
\]
The first term belongs to $\Cu(Y)$ by Step~1. Each
$\widetilde a^{(j)}$ has finite propagation, and
$q\in\Cu(Y)$ by Lemma~\ref{lem:q-in-roe}. Hence
$UaU^*\in\Cu(Y)$.

Finally, every $a\in\Cu(X)$ is a norm limit of
finite-propagation operators. Since conjugation by $U$
preserves the operator norm, the conclusion follows for
all $a\in\Cu(X)$.
\end{proof}

Lemmas~\ref{lem:compression} and~\ref{lem:lifting} give the
following description of the corner determined by $q$.

\begin{prop}\label{prop:corner-identity}
Conjugation by $U$ defines a $*$-isomorphism
\[
\Phi\colon\Cu(X)\longrightarrow q\Cu(Y)q,
\qquad
\Phi(a)=UaU^*.
\]
\end{prop}

\begin{proof}
Since $U^*U=I$, the map $\Phi$ is an injective
$*$-homomorphism. Lemma~\ref{lem:lifting} and the identity
$qU=U$ show that its range is contained in $q\Cu(Y)q$.

Conversely, if $T\in\Cu(Y)$, then
$U^*TU\in\Cu(X)$ by Lemma~\ref{lem:compression}, and
$\Phi(U^*TU)=qTq$. Thus $\Phi$ is onto.
Since $q\in\Cu(Y)$ by Lemma~\ref{lem:q-in-roe},
the corner $q\Cu(Y)q$ is a hereditary $C^*$-subalgebra of
$\Cu(Y)$.
\end{proof}

The growth of the fibers forces every operator in this
corner to be a ghost.

\begin{prop}\label{prop:ghost}
Every operator in $q\Cu(Y)q$ is a ghost, and $q$ is noncompact.
\end{prop}

\begin{proof}
For $y\in F_x\subseteq Y_n$, equation~\eqref{eq:isometry-U}
gives $q\delta_y=M_n^{-1/2}\xi_x$, so
$\|q\delta_y\|=M_n^{-1/2}$. Hence, for $T\in\Cu(Y)$,
$y\in Y_n$, and $z\in Y_m$,
\begin{equation}\label{eq:ghost-corner-bound}
\bigl|\langle qTq\delta_z,\delta_y\rangle\bigr|
\leq
\|T\|\,\|q\delta_z\|\,\|q\delta_y\|
=
\frac{\|T\|}{\sqrt{M_nM_m}}.
\end{equation}

Given $\varepsilon>0$, choose $N$ such that
$\|T\|/\sqrt N<\varepsilon$, and set
$E_N=\bigsqcup_{n<N}Y_n$. This is a finite subset of $Y$.
By \eqref{eq:bundle-parameters}, one has $M_n\geq n$ for
every $n$. Thus \eqref{eq:ghost-corner-bound} is at most
$\|T\|/\sqrt N<\varepsilon$ whenever at least one of $y,z$
lies outside $E_N$. This proves that $qTq$ is a ghost.

Finally, $q\xi_x=\xi_x$ for every $x\in X$.
The family $(\xi_x)_{x\in X}$ is infinite and orthonormal,
so $q$ has infinite rank and is not compact.
\end{proof}

There is also a noncompact block-rank-one ghost projection
on a sparse subspace of $Y$. Choose $x_n\in C_{L_n}$ for
each $n$, and equip
$Z=\bigsqcup_{n\geq1}F_{x_n}\subseteq Y$ with the subspace
metric. Since $F_{x_n}\subseteq Y_n$, the separation of
the components $Y_n$ implies that
$d_Y(F_{x_n},F_{x_m})\to\infty$ as $n+m\to\infty$ with
$n\ne m$. Hence $Z$ is sparse.

The set $Z$ is a union of whole fibers, so $\chi_Z$ commutes
with $q$. The compression
$p=(\chi_Zq\chi_Z)|_{\ell^2(Z)}$ is therefore a projection;
more precisely, it is the orthogonal direct sum of the
rank-one projections onto $\mathbb C\xi_{x_n}$.
Compressing finite-propagation approximants to $q$ shows
that $p\in\Cu(Z)$, because the metric on $Z$ is inherited
from $Y$. Equation~\eqref{eq:ghost-corner-bound}, applied
with $T=I$, shows that $p$ is a ghost. Its range contains
the infinite orthonormal family $(\xi_{x_n})_{n\geq1}$,
so it is noncompact. Its extension by zero is
$\chi_Zq\chi_Z\in\Cu(Y)$, which is also a block-rank-one
ghost projection.

We can now complete the proof of Theorem~\ref{thm:main}.

\begin{proof}
Take $X,Y$ from \eqref{eq:X-and-Y}, and let $q=UU^*$,
where $U$ is defined by \eqref{eq:isometry-U}.
Proposition~\ref{prop:corner-identity} identifies $\Cu(X)$
with the hereditary corner $q\Cu(Y)q$, and
Proposition~\ref{prop:ghost} shows that every operator in
this corner is a ghost.

Proposition~\ref{prop:no-embedding} rules out a coarse
embedding from $X$ into $Y$.
By Proposition~\ref{prop:additional-geometry}, the space
$X$ has asymptotic dimension one and admits a bi-Lipschitz
embedding into a Hilbert space, whereas $Y$ admits no
coarse embedding into a Hilbert space.
A coarse embedding from $Y$ into $X$ would therefore give,
by composition, a coarse embedding of $Y$ into a Hilbert
space, a contradiction.
\end{proof}

\section{Injective coarse maps and embeddings}
\label{sec:positive}

To prove Theorem~\ref{thm:compact-ghost-rigidity}, we first use
the hypothesis on ghost projections to obtain a coarse embedding
from the matrix coefficients of a spatial implementation.
We then apply Hall's theorem to find an injective map at uniformly
bounded distance from this coarse embedding.

Throughout this section, $X$ and $Y$ are nonempty uniformly
locally finite metric spaces; in particular, both are countable.
The case $X=\varnothing$ is immediate.

\subsection{Coefficient estimates and coarse control}

A family $\mathcal P\subseteq\Cu(Y)$ is
\emph{equi-approximable} if, for every $\varepsilon>0$, there
is $R<\infty$ such that each $p\in\mathcal P$ admits an
approximation $b_p\in B(\ell^2(Y))$ with
$\operatorname{prop}(b_p)\leq R$ and
$\|p-b_p\|<\varepsilon$.

We begin with a spatial implementation theorem of Braga,
Farah, and Vignati.

\begin{lem}[{\cite[Lemma~6.1]{BFV20}}]\label{lem:spatial-implementation}
Let $\Phi\colon\Cu(X)\to\Cu(Y)$ be an injective
$*$-homomorphism with hereditary range. There is an isometry
$U\colon\ell^2(X)\to\ell^2(Y)$ such that
\[
\Phi(a)=UaU^*
\qquad(a\in\Cu(X)).
\]
\end{lem}

For $\Phi$ and $U$ as in
Lemma~\ref{lem:spatial-implementation}, put $q=\Phi(1)$.
Then $q=UU^*$ and $\Phi(\Cu(X))=q\Cu(Y)q$.
Indeed, the range of $\Phi$ has unit $q$ and is therefore
contained in $q\Cu(Y)q$. Conversely, every positive
$b\in q\Cu(Y)q$ satisfies $0\leq b\leq\|b\|q$, so
heredity implies that $b$ belongs to the range of $\Phi$.

We will use the following uniform approximation result of
Braga and Farah~\cite[Lemma~4.9]{BF21}, see also Braga, Farah, and Vignati~\cite[Lemma~2.3]{BFV24}.

\begin{lem}[{\cite[Lemma~4.9]{BF21}}]\label{lem:uniform-subseries}
Let $M$ be a uniformly locally finite metric space, let $I$
be a countable set, and let $(p_i)_{i\in I}$ be pairwise
orthogonal projections on $\ell^2(M)$. For $A\subseteq I$,
let $p_A=\sum_{i\in A}p_i$, where the sum is taken in the
strong operator topology. If $p_A\in\Cu(M)$ for every
$A\subseteq I$, then the family
$\{p_A:A\subseteq I\}$ is equi-approximable.
\end{lem}

The absence of block-rank-one ghost projections gives a
uniform lower bound for the coefficients of equi-approximable
rank-one projections. The following is the rank-one case of
Braga, Farah, and Vignati~\cite[Lemma~3.5(2)]{BFV24}.

\begin{lem}[{\cite[Lemma~3.5(2)]{BFV24}}]\label{lem:rank-one-coefficients}
Suppose that $\Cu(Y)$ contains no block-rank-one ghost
projection. If $\mathcal P$ is a nonempty equi-approximable
family of rank-one projections in $\Cu(Y)$, then
\[
\inf_{p\in\mathcal P}\sup_{y\in Y}\|p\delta_y\|>0.
\]
\end{lem}

For $\xi\in\ell^2(Y)$, write
$\|\xi\|_\infty=\sup_{y\in Y}|\langle\xi,\delta_y\rangle|$.
If $U\colon\ell^2(X)\to\ell^2(Y)$ is an isometry, then
$Ue_{xx}U^*$ is the orthogonal projection onto
$\mathbb C U\delta_x$. Since $U\delta_x$ is a unit vector,
for every $x\in X$ and $y\in Y$ we have
\begin{equation}\label{eq:rank-one-matrix-coefficients}
\|Ue_{xx}U^*e_{yy}\|
=
\|Ue_{xx}U^*\delta_y\|
=
|\langle U\delta_x,\delta_y\rangle|.
\end{equation}

Lemmas~\ref{lem:uniform-subseries}
and~\ref{lem:rank-one-coefficients} now give the coefficient
bound needed to choose a map from $X$ to $Y$.

\begin{lem}\label{lem:uniform-coefficient}
Suppose that $\Cu(Y)$ contains no block-rank-one ghost
projection. If $U\colon\ell^2(X)\to\ell^2(Y)$ is an
isometry satisfying $U\Cu(X)U^*\subseteq\Cu(Y)$, then
\[
\inf_{x\in X}\|U\delta_x\|_\infty>0.
\]
\end{lem}

\begin{proof}
For $x\in X$, put $p_x=Ue_{xx}U^*$. Since $U$ is an
isometry, the operators $p_x$ are pairwise orthogonal
rank-one projections.

For every $A\subseteq X$, the diagonal projection $\chi_A$
belongs to $\Cu(X)$, and the finite partial sums of
$\sum_{x\in A}e_{xx}$ converge strongly to $\chi_A$.
Multiplication by $U$ and $U^*$ therefore gives
\begin{equation}\label{eq:diagonal-projection-sums}
p_A:=\sum_{x\in A}p_x
=
U\chi_AU^*
\in\Cu(Y),
\qquad A\subseteq X,
\end{equation}
with the sum again taken in the strong operator topology.

Since $X$ is countable,
\eqref{eq:diagonal-projection-sums} allows us to apply
Lemma~\ref{lem:uniform-subseries}. Thus
$\{p_A:A\subseteq X\}$ is equi-approximable, and so is its
subfamily $\{p_x:x\in X\}$. Applying
Lemma~\ref{lem:rank-one-coefficients} to this subfamily and
using \eqref{eq:rank-one-matrix-coefficients} proves the
assertion.
\end{proof}

The coarse control results of Braga, Farah, and
Vignati~\cite[Lemmas~5.3 and~6.2]{BFV20} are stated in
terms of the quantities $\|\Phi(e_{xx})e_{yy}\|$.
For $\Phi(a)=UaU^*$, equation
\eqref{eq:rank-one-matrix-coefficients} identifies these
quantities with the matrix coefficients of $U$.
Their results therefore take the following form.

\begin{lem}\label{prop:coefficient-selector-coarse}
Let $U\colon\ell^2(X)\to\ell^2(Y)$ be an isometry satisfying
$U\Cu(X)U^*\subseteq\Cu(Y)$. Suppose that a map
$f\colon X\to Y$ and a number $\delta>0$ satisfy
\begin{equation}\label{eq:selector-coefficient}
|\langle U\delta_x,\delta_{f(x)}\rangle|
\geq\delta
\qquad(x\in X).
\end{equation}
Then $f$ is coarse. If $U\Cu(X)U^*$ is hereditary in
$\Cu(Y)$, then $f$ is a coarse embedding.
\end{lem}

Under the hypotheses of Lemma~\ref{lem:uniform-coefficient},
a map satisfying \eqref{eq:selector-coefficient} can be
chosen by taking
$0<\delta<\inf_{x\in X}\|U\delta_x\|_\infty$ and, for each
$x\in X$, selecting a coordinate at which
$|U\delta_x|$ exceeds $\delta$.
For a hereditary embedding,
Lemma~\ref{prop:coefficient-selector-coarse} therefore
provides the coarse embedding to which we will apply
Hall's theorem.

\subsection{Hall's theorem and injectivity}

We now show that a map satisfying \eqref{eq:selector-coefficient}
is at uniformly bounded distance from an injection. This requires
only the inclusion $U\ell^\infty(X)U^*\subseteq\Cu(Y)$.
The dimension argument and the application of Hall's theorem
follow Braga, Farah, and
Vignati~\cite[Lemmas~6.8 and~6.10]{BFV20}.
The additional step is to use
Lemma~\ref{lem:uniform-subseries} and averaging over signs
to obtain one radius that works for every finite subset
of $X$.

We use the following finite form of the marriage theorem of
Hall~\cite[Theorem~1]{Hall35}.

\begin{lem}\label{lem:finite-hall}
Let $E$ be a finite set, and let $(N_x)_{x\in E}$ be a family
of subsets of a set $Y$. There is an injection $g\colon E\to Y$
such that $g(x)\in N_x$ for every $x\in E$ if and only if
\[
\left|\bigcup_{x\in F}N_x\right|\geq |F|
\qquad(F\subseteq E).
\]
\end{lem}

\begin{prop}\label{prop:hall-injectivization}
Let $U\colon\ell^2(X)\to\ell^2(Y)$ be an isometry satisfying
$U\ell^\infty(X)U^*\subseteq\Cu(Y)$. Suppose that
$f\colon X\to Y$ and $\delta>0$ satisfy
\eqref{eq:selector-coefficient}. There are an injective map
$g\colon X\to Y$ and $R<\infty$ such that
\begin{equation}\label{eq:injective-map-closeness}
d_Y(g(x),f(x))\leq R
\qquad(x\in X).
\end{equation}
\end{prop}

\begin{proof}
\emph{Step 1: choosing a common propagation bound.}
Write $\xi_x=U\delta_x$ and $P_A=U\chi_AU^*$ for
$A\subseteq X$. The vectors $(\xi_x)_{x\in X}$ are
orthonormal, and $P_A$ is the projection onto their closed
linear span over $x\in A$.

By hypothesis, $P_A\in\Cu(Y)$ for every $A\subseteq X$.
Moreover, the projections $P_{\{x\}}$ are pairwise
orthogonal and have strong sum $P_A$ over each
$A\subseteq X$. Lemma~\ref{lem:uniform-subseries} therefore
provides $R\geq0$ and operators $T_A\in B(\ell^2(Y))$
such that
\begin{equation}\label{eq:uniform-projection-approximation}
\operatorname{prop}(T_A)\leq R,
\qquad
\|P_A-T_A\|<\delta/4
\qquad(A\subseteq X).
\end{equation}

\emph{Step 2: verifying Hall's inequalities.}
For a finite set $F\subseteq X$, put
$N_R(F)=\bigcup_{x\in F}B_Y(f(x),R)$.
This set is finite because $Y$ is uniformly locally finite.
We claim that
\begin{equation}\label{eq:uniform-hall}
|N_R(F)|\geq |F|
\qquad(F\subseteq X\text{ finite}).
\end{equation}

Suppose that $|N_R(F)|<|F|$ for some finite $F$.
The restriction
$\chi_{N_R(F)}\colon\operatorname{ran}(P_F)
\to\ell^2(N_R(F))$
then has a nonzero kernel, since its domain has dimension
$|F|$ and its codomain has dimension $|N_R(F)|$.
Choose a unit vector
$\zeta=\sum_{x\in F}a_x\xi_x$ in this kernel. Thus
\begin{equation}\label{eq:hall-kernel-vector}
\chi_{N_R(F)}\zeta=0,
\qquad
\sum_{x\in F}|a_x|^2=1.
\end{equation}

Let $S=f(F)$ and $e=\chi_S$. Since $S\subseteq N_R(F)$,
equation~\eqref{eq:hall-kernel-vector} gives $e\zeta=0$.
For each $x\in F$, condition~\eqref{eq:selector-coefficient}
also gives
$\|e\xi_x\|\geq
|\langle\xi_x,\delta_{f(x)}\rangle|\geq\delta$.

Average the squared norm over all choices of signs
\(\varepsilon=(\varepsilon_x)_{x\in F}\in\{-1,1\}^F\).
For distinct \(x,z\in F\), the average of
\(\varepsilon_x\varepsilon_z\) is zero, whereas
\(\varepsilon_x^2=1\).
Expanding the squared norm therefore gives
\begin{equation}\label{eq:hall-sign-average}
2^{-|F|}
\sum_{\varepsilon\in\{-1,1\}^{F}}
\left\|e\sum_{x\in F}\varepsilon_xa_x\xi_x\right\|^2
=
\sum_{x\in F}|a_x|^2\|e\xi_x\|^2
\geq\delta^2,
\end{equation}
where the last inequality uses
\eqref{eq:hall-kernel-vector}.
By \eqref{eq:hall-sign-average}, there is a choice of signs
for which the norm inside the sum is at least $\delta$.
Fix such a choice and put
$A=\{x\in F:\varepsilon_x=1\}$.
Since
$\sum_{x\in F}\varepsilon_xa_x\xi_x
=2P_A\zeta-\zeta$
and $e\zeta=0$, we obtain
\begin{equation}\label{eq:rademacher-lower-bound}
\|eP_A\zeta\|\geq\delta/2.
\end{equation}

On the other hand,
\eqref{eq:hall-kernel-vector} says that $\zeta$ is supported
in $Y\setminus N_R(F)$. If $y\in S$ and
$z\notin N_R(F)$, then $d_Y(y,z)>R$.
The propagation bound in
\eqref{eq:uniform-projection-approximation} therefore gives
$eT_A\chi_{Y\setminus N_R(F)}=0$, and hence
$eT_A\zeta=0$.
Figure~\ref{fig:hall-support} illustrates this support
separation.

\begin{figure}
\centering
\begin{tikzpicture}
\draw[fill=black!4] (0,0) ellipse (2.35 and 1.4);
\draw[dashed] (-.8,0) ellipse (.8 and .55);

\foreach \x/\y in {-1.1/.15,-.65/.25,-.95/-.2,-.4/-.1}{
  \fill (\x,\y) circle (1.6pt);
}

\node at (.75,.65) {$N_R(F)$};
\node at (-.8,-.85) {$S=f(F)$};

\draw[rounded corners=9pt,dashed]
  (3.8,-.65) rectangle (6.3,.65);
\node at (5.05,0) {$\supp(\zeta)$};

\draw[-{Stealth},thick]
  (5.05,.8)
  .. controls (5.05,2.1) and (-.8,2.1) .. (-.8,.7);
\node[fill=white,inner sep=2pt] at (2.1,1.8)
  {$\|eP_A\zeta\|\geq\delta/2$};

\draw[-{Stealth},dashed]
  (5.05,-.8)
  .. controls (5.05,-2.2) and (-.8,-2.2) .. (-.8,-1.05);
\node[fill=white,inner sep=2pt] at (2.1,-1.85)
  {$eT_A\zeta=0$};
\end{tikzpicture}
\caption{The support separation in Step~2.
The vector $\zeta$ is supported outside $N_R(F)$.
Averaging over signs produces a subset $A\subseteq F$
for which the component of $P_A\zeta$ on $S$ has norm
at least $\delta/2$. An operator of propagation at most
$R$ sends $\zeta$ to a vector whose component on $S$
is zero.}
\label{fig:hall-support}
\end{figure}

Using \eqref{eq:uniform-projection-approximation} and
$\|\zeta\|=1$, we now obtain
$\|eP_A\zeta\|
=\|e(P_A-T_A)\zeta\|
\leq\|P_A-T_A\|<\delta/4$.
This contradicts \eqref{eq:rademacher-lower-bound} and
proves \eqref{eq:uniform-hall}.

\emph{Step 3: choosing an injection on all of \(X\).}
If \(X\) is finite, the assertion follows directly from
\eqref{eq:uniform-hall} and Lemma~\ref{lem:finite-hall}.
Suppose that \(X\) is infinite, and enumerate it as
\(X=\{x_1,x_2,\ldots\}\).

For each \(n\geq1\), the same lemma gives an injection
\(h_n\colon\{x_1,\ldots,x_n\}\to Y\) such that
\(h_n(x_i)\in B_Y(f(x_i),R)\) for \(1\leq i\leq n\).
We choose infinite sets of indices
\(I_0\supseteq I_1\supseteq I_2\supseteq\cdots\),
starting with \(I_0=\mathbb N\).
Suppose that \(I_{k-1}\) has been chosen.
After discarding the finitely many indices smaller than \(k\),
the values \(h_n(x_k)\), \(n\in I_{k-1}\), lie in the finite
set \(B_Y(f(x_k),R)\).
One of these values occurs infinitely often.
Define \(g(x_k)\) to be such a value and let
\[
 I_k=\{n\in I_{k-1}:n\geq k,\ h_n(x_k)=g(x_k)\}.
\]
Then \(I_k\) is infinite.

For \(i<j\), choose \(n\in I_j\).
The inclusions \(I_j\subseteq I_i\) give
\(g(x_i)=h_n(x_i)\) and \(g(x_j)=h_n(x_j)\).
These values are distinct because \(h_n\) is injective.
Thus \(g\) is injective on \(X\).
Finally, \(g(x_k)\in B_Y(f(x_k),R)\) for every \(k\),
which proves \eqref{eq:injective-map-closeness}.
\end{proof}

The argument also yields an injective coarse map from a
rank-preserving embedding. An embedding is
\emph{rank preserving} if it sends rank-one projections
to rank-one projections. Since a finite-rank projection
is a finite sum of orthogonal rank-one projections, this
is equivalent to preserving the rank of every finite-rank
projection. Such an embedding sends compact operators to
compact operators: finite-rank operators are linear
combinations of rank-one projections, and compact operators
are their norm limits.

The strong-continuity reduction of Braga, Farah, and
Vignati~\cite[Theorem~4.3 and its proof]{BFV20}, restated
in Braga~\cite[Theorem~5.1]{Bra22}, preserves the restriction
of the embedding to the compact operators. Together with
Braga's spatial implementation
theorem~\cite[Theorem~5.3]{Bra22}, this gives the following
statement.

\begin{lem}\label{lem:strong-continuity-reduction}
Let $\Phi\colon\Cu(X)\to\Cu(Y)$ be a rank-preserving
embedding. There is an isometry
$U\colon\ell^2(X)\to\ell^2(Y)$ such that
$U\Cu(X)U^*\subseteq\Cu(Y)$ and
$UkU^*=\Phi(k)$ for every $k\in\cK(\ell^2(X))$.
\end{lem}

Braga, Farah, and Vignati~\cite[Theorem~5.4]{BFV20}
obtained a uniformly finite-to-one coarse map from a
rank-preserving embedding when every sparse subspace
of $Y$ yields only compact ghost projections.
The next corollary gives an injective coarse map under
the hypothesis of Theorem~\ref{thm:compact-ghost-rigidity}.

\begin{cor}\label{cor:spatial-injective-coarse-map}
Suppose that $\Cu(Y)$ contains no block-rank-one ghost
projection. If there is a rank-preserving embedding
$\Cu(X)\to\Cu(Y)$, then there is an injective coarse map
$X\to Y$.
\end{cor}

\begin{proof}
Lemma~\ref{lem:strong-continuity-reduction} gives an
isometry $U\colon\ell^2(X)\to\ell^2(Y)$ with
$U\Cu(X)U^*\subseteq\Cu(Y)$.
By Lemma~\ref{lem:uniform-coefficient}, we may choose
$\delta>0$ and $f\colon X\to Y$ satisfying
\eqref{eq:selector-coefficient}.
Lemma~\ref{prop:coefficient-selector-coarse} shows that
$f$ is coarse.

Proposition~\ref{prop:hall-injectivization} gives an
injection $g\colon X\to Y$ and $R<\infty$ satisfying
\eqref{eq:injective-map-closeness}.
For $x,z\in X$, we have
$d_Y(g(x),g(z))\leq d_Y(f(x),f(z))+2R$.
Thus $g$ is coarse.
\end{proof}

We now complete the proof of
Theorem~\ref{thm:compact-ghost-rigidity}.

\begin{proof}
The case $X=\varnothing$ is immediate.
If $X$ is nonempty, then $Y$ is nonempty because $\Phi$
is injective.
By Lemma~\ref{lem:spatial-implementation}, there is an
isometry $U\colon\ell^2(X)\to\ell^2(Y)$ implementing
$\Phi$.

Choose $\delta>0$ and $f\colon X\to Y$ satisfying
\eqref{eq:selector-coefficient}, using
Lemma~\ref{lem:uniform-coefficient}.
Since the range of $\Phi$ is hereditary,
Lemma~\ref{prop:coefficient-selector-coarse} shows that
$f$ is a coarse embedding.
Proposition~\ref{prop:hall-injectivization} then gives
an injection $g\colon X\to Y$ satisfying
\eqref{eq:injective-map-closeness} for some $R<\infty$.

Let $\rho_-$ and $\rho_+$ be lower and upper control
functions for $f$. By
\eqref{eq:injective-map-closeness} and the triangle
inequality, the functions
\[
t\longmapsto\max\{0,\rho_-(t)-2R\},
\qquad
t\longmapsto\rho_+(t)+2R
\]
are lower and upper control functions for $g$.
The lower control tends to infinity, so $g$ is an
injective coarse embedding.
\end{proof}

The following implication was observed by Braga, Farah, and
Vignati~\cite[p.~658]{BFV24}.
We include the proof to explain how to pass from a block
decomposition to a sparse subspace.

\begin{lem}\label{lem:sparse-implies-block-rank-one}
If every sparse subspace of $Y$ yields only compact
ghost projections, then $\Cu(Y)$ contains no
block-rank-one ghost projection.
\end{lem}

\begin{proof}
Suppose that $p\in\Cu(Y)$ is a block-rank-one ghost
projection, with blocks $(Y_n)_{n\geq1}$ as in
Definition~\ref{def:block-rank-one}.

We first choose a sparse union of these blocks.
Every neighborhood of finite radius of a finite subset
of $Y$ is finite. Since the sets $Y_n$ are pairwise
disjoint, only finitely many of them meet such a
neighborhood.
We can therefore choose
$n_1<n_2<\cdots$ inductively so that
\begin{equation}\label{eq:sparse-block-separation}
d_Y\left(Y_{n_k},\bigcup_{j<k}Y_{n_j}\right)>k
\qquad(k\geq2).
\end{equation}
Indeed, at the $k$th stage, choose $n_k>n_{k-1}$ so
that $Y_{n_k}$ avoids the closed $k$-neighborhood of
$\bigcup_{j<k}Y_{n_j}$.
By \eqref{eq:sparse-block-separation}, the distances
between distinct selected blocks tend to infinity
as the sum of their indices tends to infinity.
Thus $Z=\bigsqcup_{k\geq1}Y_{n_k}$ is a sparse
subspace of $Y$.

The projection $p$ is block diagonal with respect to
the sets $Y_n$, and $Z$ is a union of whole blocks.
Hence $\chi_Z$ commutes with $p$, and the compression
$p_Z=(\chi_Zp\chi_Z)|_{\ell^2(Z)}$ is a projection.
It is the direct sum of the rank-one projections
on the selected blocks, so it has infinite rank.

Compressing finite-propagation approximants to $p$
shows that $p_Z\in\Cu(Z)$: compression does not increase
the operator norm, and the metric on $Z$ is inherited
from $Y$. Its matrix coefficients are the corresponding
coefficients of $p$, so $p_Z$ is also a ghost.
This is a noncompact ghost projection in $\Cu(Z)$,
contrary to the hypothesis.
\end{proof}

Lemma~\ref{lem:sparse-implies-block-rank-one} and
Theorem~\ref{thm:compact-ghost-rigidity} prove
Corollary~\ref{cor:sparse-ghost-rigidity}.
If every ghost projection in $\Cu(Y)$ is compact,
then $\Cu(Y)$ contains no block-rank-one ghost
projection, since every block-rank-one projection
has infinite rank.
Theorem~\ref{thm:compact-ghost-rigidity} therefore also
gives Corollary~\ref{cor:global-ghost-rigidity}.

\end{document}